\documentclass[oneeqnum,onethmnum,onefignum]{siamltex}
\usepackage{datetime}
\newcommand{\timestamp}{\currenttime, \today}

\newcommand{\K}[1]{{\cal K}_{#1}}
\newcommand{\xG}[1]{x^G_{#1}}
\newcommand{\xM}[1]{x^M_{#1}}
\newcommand{\xE}[1]{x^E_{#1}}
\newcommand{\xR}[1]{x^R_{#1}}
\newcommand{\yE}[1]{y^E_{#1}}

\newcommand{\argmin}{\mathop{\rm argmin}}
\newcommand{\SPAN}{\mathop{\rm span}}
\newcommand{\norm}[1]{\|#1\|}
\newcommand{\ceil}[1]{\lceil#1\rceil}
\newcommand{\floor}[1]{\lfloor#1\rfloor}
\newcommand{\forany}{\mbox{ for any }}

\newcommand{\FOR}{\textbf{for}}
\newcommand{\STEP}{\textbf{step}}
\newcommand{\UNTIL}{\textbf{until}}

\begin{document}
\pagestyle{plain}
\bibliographystyle{siam}

\title {Equivalence of Krylov Subspace Methods \\
	for Skew-Symmetric Linear Systems}

\author{Stanley C. Eisenstat\thanks{%
	Department of Computer Science, Yale University,
	P.~O. Box 208285,
	New Haven, CT 06520-8285.
	(\timestamp)
	\vspace{2em}}}

\maketitle
\begin{abstract}
  In recent years two Krylov subspace methods have been
  proposed for solving skew symmetric linear systems,
  one based on the minimum residual condition, the other
  on the Galerkin condition.
  We give new, algorithm-independent proofs that in
  exact arithmetic the iterates for these methods are
  identical to the iterates for the conjugate gradient
  method applied to the normal equations and the classic
  Craig's method, respectively, both of which select
  iterates from a Krylov subspace of lower dimension.
  More generally, we show that projecting an approximate
  solution from the original subspace to the lower-dimensional
  one cannot increase the norm of the error or
  residual.
\end{abstract}



\section{Introduction}

Consider the system of linear equations
\begin{equation} \label{eqn:Axb}
  A x = b
\end{equation}
where the $n \times n$ coefficient matrix $A$ is real,
skew symmetric (i.e., $A^t = -A$), and nonsingular (so
that $n$ is even).
Krylov subspace methods for solving (\ref{eqn:Axb})
that are based directly on $A$ and $b$ compute a
sequence $\{ x_m \}$ of approximate solutions
where\footnote{
  For simplicity we take $x_0 = 0$.
}
\[
  x_m \in \SPAN \{b, Ab, \ldots, A^{m-1}b \}
      \equiv \K{m}(A,b) .
\]
The iterate $x_m$ is often the unique vector that
satisfies either the \emph{Galerkin condition}
\begin{equation} \label{eqn:g}
  p^t (b - A \xG{m}) = 0,
  \quad \forany p \in \K{m}(A,b),
\end{equation}
or the \emph{minimum residual condition}
\[
  \xM{m} = \argmin_{z \in \K{m}(A,b)} \norm{b - Az} ,
\]
where $\norm{\cdot}$ denotes the Euclidean norm.
The latter is easily seen to be equivalent to
\begin{equation} \label{eqn:mr}
  (Ap)^t (b - A \xM{m}) = 0,
  \quad \forany p \in \K{m}(A,b) .
\end{equation}

A classic approach to solving (\ref{eqn:Axb}) is the
conjugate gradient method (itself a Krylov subspace
method based on the Galerkin condition) applied to
the normal equations, either $A A^t y = b$ or
$A^t A x = A^t b$.

CGNE~\cite{hestenes},\cite[p.~105]{greenbaum} (also
known as Craig's method~\cite{craig};
see Figure~\ref{fig:alg}) uses CG
to solve $A A^t y = b$ and sets $x = A^t y$.
Thus the iterate $\xE{q}$ is the unique vector
satisfying
\[
 \xE{q} = A^t \yE{q} \in A^t \K{q}(A A^t,b)
		       = \K{q}(A^t A,A^t b)
\]
and
\[
  p^t (b - A \xE{q}) = p^t (b - A A^t \yE{q}) = 0,
  \quad \forany p \in \K{q}(A A^t,b) .
\]
Since $A$ is skew symmetric, this can be written as
$\xE{q} \in \K{q}(A^2,Ab)$ and
\begin{equation} \label{eqn:cgne}
  p^t (b - A \xE{q}) = 0,
  \quad \forany p \in \K{q}(A^2,b) .
\end{equation}
Moreover, it follows
that~\cite{hestenes},\cite[p.~106]{greenbaum}
\[
  \xE{q} = \argmin_{z \in \K{q}(A^2,Ab)} \norm{z - x} .
\]

\begin{figure} \label{fig:alg} \centering \footnotesize
\begin{minipage}{5in}
 \begin{tabbing}
   xx \= xx \= xx \= xx \= xx \= xx \= xx \= \kill
   $r_0 = b$; $p_0 = A^t r_0$ \\
   \FOR~$q = 1$ \STEP~$1$ \UNTIL~convergence \\[.1em]
\>   $\alpha_q =
       \left\{ \begin{array}{@{}ll@{}}
	 \norm{r_{q-1}}^2     / \norm{p_{q-1}}^2   & (E) \\
	 \norm{A^t r_{q-1}}^2 / \norm{A p_{q-1}}^2 & (R) \\
       \end{array} \right.$ \\[.1em]
\>   $x_q = x_{q-1} + \alpha_q p_{q-1}$   \\
\>   $r_q = r_{q-1} - \alpha_q A p_{q-1}$ \\[.1em]
\>   $\beta_q =
       \left\{ \begin{array}{@{}ll@{}}
	 \norm{r_{q}}^2     / \norm{r_{q-1}}^2     & (E) \\
	 \norm{A^t r_{q}}^2 / \norm{A^t r_{q-1}}^2 & (R) \\
       \end{array} \right.$ \\[.3em]
\>   $p_q = A^t r_q + \beta_q p_{q-1}$
 \end{tabbing}
\end{minipage}
\qquad \qquad
\begin{minipage}{5in}
 \begin{tabbing}
   xx \= xx \= xx \= xx \= xx \= xx \= xx \= \kill
   $r_0 = b$; $p_0 = -A r_0$ \\
   \FOR~$q = 1$ \STEP~$1$ \UNTIL~convergence \\[.1em]
\>   $\alpha_q =
       \left\{ \begin{array}{@{}ll@{}}
	 \norm{r_{q-1}}^2   / \norm{p_{q-1}}^2   & (E) \\
	 \norm{A r_{q-1}}^2 / \norm{A p_{q-1}}^2 & (R) \\
       \end{array} \right. $\\[.1em]
\>   $x_q = x_{q-1} + \alpha_q p_{q-1}$   \\
\>   $r_q = r_{q-1} - \alpha_q A p_{q-1}$ \\[.1em]
\>   $\beta_q =
       \left\{ \begin{array}{@{}ll@{}}
	 \norm{r_{q}}^2   / \norm{r_{q-1}}^2   & (E) \\
	 \norm{A r_{q}}^2 / \norm{A r_{q-1}}^2 & (R) \\
       \end{array} \right.$ \\[.3em]
\>   $p_q = -A r_q + \beta_q p_{q-1}$
 \end{tabbing}
\end{minipage}
\caption{CGNE (E) and CGNR (R) for general (left) and
	 skew symmetric (right) systems.}
\end{figure}

CGNR~\cite{hestenes-stiefel},\cite[p.~105]{greenbaum}
(also known as
CGLS~\cite{paige-saunders}; see Figure~\ref{fig:alg})
uses CG to solve $A^t A x = A^t b$.
Thus the iterate $\xR{q}$ is the unique vector
satisfying $\xR{q} \in \K{q}(A^t A,A^t b)$ and
\[
  (Ap)^t (b - A \xR{q}) = p^t (A^t b - A^t A \xR{q}) = 0,
  \quad \forany p \in \K{q}(A^t A,A^t b) .
\]
Since $A$ is skew symmetric, this can be written as
$\xR{q} \in \K{q}(A^2,Ab)$ and
\begin{equation} \label{eqn:cgnr}
  (Ap)^t (b - A \xR{q}) = 0,
  \quad \forany p \in \K{q}(A^2,Ab) .
\end{equation}
Moreover, it follows
that~\cite{hestenes-stiefel},\cite[pp.~105--6]{greenbaum}
\[
  \xR{q} = \argmin_{z \in \K{q}(A^2,Ab)} \norm{b - Az} .
\]

CGNE and CGNR are often disparaged\footnote{
  Greenbaum~\cite[p.~106]{greenbaum} rebuts this view.}
since they square the condition number (which may slow
convergence) and may be more susceptible to round-off
error (which is why the algorithms in
Figure~\ref{fig:alg} avoid multiplication by $A A^t$ and
$A^t A$, respectively).

Thus in recent years several authors have derived Krylov
subspace methods that solve (\ref{eqn:Axb}) directly.
Gu and Qian~\cite{gu-qian} and Greif and
Varah~\cite{greif-varah} impose the Galerkin\footnote{
  Gu and Qian~\cite{gu-qian} claim incorrectly that they
  are imposing the minimum residual condition.
}
condition (\ref{eqn:g}) on the subspace $\K{m}(A,b)$;
while Jiang~\cite{jiang}, Idema and
Vuik~\cite{idema-vuik}, and Greif and
Varah~\cite{greif-varah} impose the minimum residual
condition (\ref{eqn:mr}).
Greif and Varah~\cite{greif-varah} show that the odd
iterates $\xG{2q+1}$ do not exist, that their
algorithm for the even iterates $\xG{2q}$ is equivalent
to CGNE, and that $\xM{2q+1} = \xM{2q}$.

In this paper we give new, algorithm-independent proofs
that $\xG{2q} = \xE{q}$ and that\footnote{
  That $\xM{2q} = \xR{q}$ also follows from the
  observation (see~\cite[\S2.4]{idema-vuik}) that the
  Huang, Wathen, and Li~\cite{huang-wathen-li}
  algorithm, which computes only the even iterates
  $\xM{2q}$, is equivalent to CGNR.}
$\xM{2q+1} = \xM{2q} = \xR{q}$.
More generally we show that any approximate solution
$z$ that belongs to $\K{m}(A,b)$ but not to
$\K{\floor{m/2}}(A^2,Ab)$ has a larger error
$\norm{z-x}$ and residual $\norm{b-Az}$ than its
projection onto the lower-dimensional subspace.
Thus there does not seem to be any advantage to seeking
an approximate solution in $\K{m}(A,b)$.

\section{Main results}
We begin with a simple consequence of skew symmetry.

\begin{lemma} \label{lemma}
If $A$ is skew symmetric, the subspaces $\K{s}(A^2,Ab)$
and $\K{t}(A^2,b)$ are orthogonal and the solution $x$
of $Ax = b$ is orthogonal to $\K{t}(A^2,b)$, for any
$s, t \geq 0$.
\end{lemma}
\begin{proof}
Without loss of generality it suffices to show that
both $(A^2)^k Ab$ and $x$ are orthogonal to $(A^2)^\ell b$
for any $0 \leq k < s$ and $0 \leq \ell < t$.
But
\[
  \left( (A^2)^k Ab \right)^t \! \left((A^2)^\ell b \right)
  = \left( A^{2k+1} b \right)^t \! \left( A^{2\ell} b \right)
  = (-1)^{k+\ell+1} \! \left( A^{k+\ell} b \right)^t \!
	A \left( A^{k+\ell} b \right)
  = 0
\]
and
\[
  x^t \! \left( (A^2)^\ell b \right)
  = x^t \! \left( A^{2\ell} Ax \right)
  = (-1)^\ell \left( A^\ell x \right)^t \! A \left( A^\ell x \right)
  = 0
\]
since $z^t A z = 0$ for any $z$.
\end{proof}

By grouping even and odd powers of $A$,
any $p \in \K{m}(A,b)$ can be written as
$p = p_e + p_o$ for some $p_e \in \K{q_e}(A^2,b)$ and
$p_o \in \K{q_o}(A^2,Ab)$, where $q_e = \ceil{m/2}$ and
$q_o = \floor{m/2}$.
By Lemma~\ref{lemma} we have that $p_e$ is orthogonal to
$p_o$.

\begin{theorem} \label{thm:equal}
If $A$ is skew symmetric, the Galerkin iterates
$\{\xG{m}\}$ and the CGNE iterates $\{\xE{q}\}$
satisfy $\xG{2q} = \xE{q}$; and the minimum residual
iterates $\{\xM{m}\}$ and the CGNR iterates $\{\xR{q}\}$
satisfy $\xM{2q+1} = \xM{2q} = \xR{q}$.
\end{theorem}
\begin{proof}
($\xG{2q} = \xE{q}$):
Since $\xE{q} \in \K{q}(A^2,Ab) \subseteq \K{2q}(A,b)$,
by the Galerkin condition (\ref{eqn:g}) it suffices to
prove that $\xE{q}$ satisfies
\[
  p^t (b - A \xE{q}) = 0,
  \quad \forany p \in \K{2q}(A,b) .
\]
Any $p \in \K{2q}(A,b)$ can be written as $p = p_e + p_o$
as above.
Since $p_e \in \K{q}(A^2,b)$,
\[
  p^t (b - A\xE{q}) = p_e^t (b - A\xE{q}) + p_o^t (b - A\xE{q})
		    =                       p_o^t (b - A\xE{q})
\]
by (\ref{eqn:cgne}).
But since $p_o \in \K{q}(A^2,Ab)$ and
\[
  b - A\xE{q} \in b + A \K{q}(A^2,Ab) \subseteq \K{q+1}(A^2,b) ,
\]
we have that $p_o$ is orthogonal to $b-A\xE{q}$ by
Lemma~\ref{lemma} and so $p^t (b - A\xE{q}) = 0$.

($\xM{2q+1} = \xM{2q} = \xR{q}$):
Note that
$\K{q}(A^2,Ab) \subseteq \K{2q}(A,b) \subseteq \K{2q+1}(A,b)$.
Thus $\xR{q} \in \K{2q+1}(A,b)$ and
$\xR{q} \in \K{2q}(A,b)$;
and by the minimum residual condition (\ref{eqn:mr}) it
suffices to prove that $\xR{q}$ satisfies
\[
  (Ap)^t (b - A \xR{q}) = 0,
  \quad \forany p \in \K{2q+1}(A,b) ,
\]
for then
\[
  (Ap)^t (b - A \xR{q}) = 0,
  \quad \forany p \in \K{2q}(A,b)
\]
as well.
Any $p \in \K{2q+1}(A,b)$ can be written as
$p = p_e + p_o$ as above.
Since $p_o \in \K{q}(A^2,Ab)$,
\[
  (Ap)^t (b - A\xR{q})
  = (Ap_e)^t (b - A\xR{q}) + (Ap_o)^t (b - A\xR{q})
  = - p_e^t \left( A (b - A\xR{q}) \right)
\]
by (\ref{eqn:cgnr}).
But since $p_e \in \K{q+1}(A^2,b)$ and
\[
  A (b - A\xR{q}) \in Ab + A^2 \K{q}(A^2,Ab) \subseteq \K{q+1}(A^2,Ab) ,
\]
we have that $p_e$ is orthogonal to $A(b-A\xR{q})$
and so $(Ap)^t (b - A\xR{q}) = 0$.
\end{proof}

Finally we show that the extra dimensions in
$\K{m}(A,b)$ versus $\K{\floor{m/2}}(A^2,Ab)$ can not
decrease the norm of the error or the residual.

\begin{theorem} \label{thm:nobetter}
Let $z \in \K{m}(A,b)$ and write $z = z_e + z_o$ for
some $z_e \in \K{q_e}(A^2,b)$ and $z_o \in
\K{q_o}(A^2,Ab)$, where $q_e = \ceil{m/2}$ and $q_o =
\floor{m/2}$.
If $A$ is skew symmetric, the solution $x$ of $Ax = b$
satisfies
\[
  \norm{z-x}^2 = \norm{z_o-x}^2 + \norm{z_e}^2
  \mbox{\quad and \quad}
  \norm{b-Az}^2 = \norm{b-Az_o}^2 + \norm{Az_e}^2 .
\]
\end{theorem}
\begin{proof}
Since $z_o \in \K{q_o}(A^2,Ab)$, we have $z_o$ and
$x$ orthogonal to $z_e \in \K{q_e}(A^2,b)$ by
Lemma~\ref{lemma}.
Similarly, since
\[
  b-Az_o \in b + A \K{q_o}(A^2,Ab) \subseteq \K{q_o+1}(A^2,b)
\]
and
\[
  Az_e \in A \K{q_e}(A^2,b) = \K{q_e}(A^2,Ab) ,
\]
we have $b-Az_o$ orthogonal to $Az_e$.
Now apply the Pythagorean Theorem.
\end{proof}

\section{Conclusions}
Theorem~\ref{thm:nobetter} shows that there is no
advantage to using all of $\K{m}(A,b)$, and
Theorem~\ref{thm:equal} shows that CGNE and CGNR compute
the Galerkin and minimum residual iterates, at least in
exact arithmetic.\footnote{
  Paige and Saunders' LSQR is a more stable equivalent to
  CGNR if round-off error is an
  issue~\cite{paige-saunders}.
}
Thus a Krylov subspace method based on $\K{m}(A,b)$
would have to be at least as efficient and/or accurate to
warrant consideration.

Normally a Krylov subspace method is applied to a
preconditioned system
\begin{equation} \label{eqn:precon}
  \tilde{A} \tilde{x} \equiv (M_L^{-1} A M_R^{-1}) (M_R x)
    = (M_L^{-1} b) \equiv \tilde{b} .
\end{equation}
Greif and Varah~\cite{greif-varah} derive a
preconditioner (i.e., an $M_L$ and an $M_R$) for which
$\tilde{A}$ is skew symmetric, but many preconditioners
do not have this property and CGNE and CGNR
applied to (\ref{eqn:precon}) do not require it.
Thus a preconditioner for $A$ that does preserve skew
symmetry in $\tilde{A}$ would have to be at least as
efficient and/or accurate as the best general
preconditioner used with CGNE or CGNR / LSQR to warrant
consideration.

\section*{Acknowledgment}
The author thanks Chris Paige
for suggestions that improved the exposition.


\section*{References}
\bibliography{skew}
\end{document}